\documentclass[12pt,a4paper]{article}

\usepackage{amsmath,amssymb,latexsym,amsthm, amsfonts, dsfont, array}

\usepackage{comment}
\usepackage{graphics}
\usepackage{longtable}
\usepackage{slashbox}
\usepackage{multirow}
\usepackage{float}
\usepackage{color}
\usepackage{cite}
\usepackage{tabularx}
\usepackage{epsfig}
\usepackage{graphicx}

\newcommand{\NN}{\mathbb{N}}

\newcommand{\RR}{\mathbb{R}}
\newcommand{\ZZ}{\mathbb{Z}}

\newtheorem{definition}{\sc Definition}[section]
\newtheorem{teo}{\sc Theorem}[section]

\newtheorem{lemma}{\sc Lemma }[section]

\newcommand{\vect}[1]{\boldsymbol{#1}}

\newcommand{\dem}{\noindent {\it Proof.} \mbox{}}    
\newcommand{\lqqd}{\hfill $\square$}      

\begin{document}

\title{Homogenized moderately wrinkled shell theory from $3D$ Koiter's linear elasticity}

\author{Pedro Hern\'{a}ndez-Llanos\footnote{Instituto de Ciencias de la Ingenier\'ia, Universidad de O'Higgins, Rancagua, Chile. E-mail: pedro.hernandez@postdoc.uoh.cl}
\and Rajesh Mahadevan\footnote{Departamento de Matem\'atica,
Facultad de Ciencias F\'isicas y Matem\'aticas, Universidad de Concepci\'on, Casilla 160-C, Concepci\'on, Chile.
 E-mail: rmahadevan@udec.cl} \and Ravi Prakash\footnote{
 Departamento de Matem\'atica,
Facultad de Ciencias F\'isicas y Matem\'aticas, Universidad de Concepci\'on, Casilla 160-C, Concepci\'on, Chile. E-mail: rprakash@udec.cl}}

\date{\today}

\maketitle{}


\begin{abstract}
In this paper we derive, by two$-$scale convergence,  periodically wrinked shell models starting from three dimensional linear elasticity, depending of the behaviour of the small parameter $\varepsilon>0$ and $p>1$, differents theories appear. We assume that the mid-surface of the shell is given by $\displaystyle \psi(x_1,x_2)+\varepsilon^p\theta\left(\frac{x_1}{\varepsilon},\frac{x_2}{\varepsilon}\right)\vect{a}_{3}(x_1,x_2)$, where $\theta$ is $[0,1)^2$-periodic function and $p=2$. We also assume that the strain energy of the shell has the Koiter's model. 


\end{abstract}


\section{Introduction}

\medskip
 The importance of this model is that we generalize the moderatelly wrinkled plate theory for Koiter's model, and  in the case when we do not consider wrinkles on the surface we obtain the Koiter's classical plate theory.

\medskip

Throughout the paper $\overline{A}$ or $\{A\}^{-}$ denotes the closure of the set. By a domain we call a bounded open set with Lipschitz boundary. $\vect{I}$ denotes the identity matrix, by $SO(3)$ we donte the rotations in $\RR^3$ and by $so(3)$ the set of antisymmetric matrices $3\times 3$. By $\RR^{n\times n}_{sym}$ we deonte the set of symmetric matrices of the dimension $n\times n$. $\vect{e_1},\vect{e_2}, \vect{e_3}$ are the vectors of the canonical base in $\RR^3$. $\to$ denotes the strong convergence and $\rightharpoonup$ the weak convergence. By $\vect{A}\cdot\vect{B}$ we denote $tr\left(\vect{A}^{T}\vect{B}\right)$. We suppose that the Greek indices $\alpha,\beta, \varrho$ take tha values in the set $\{1,2\}$ while the Latin indices $p,i,j$ take values in the set $\{1,2,3\}$.

\vspace{0.5 cm}

\tableofcontents

\section{Setting up the moderately wrinkled shell problem}   
We begin by introducing some further notation. Set $Y=[0,1)^2$ and $\mathcal{Y}=\RR^2/\ZZ^2$. For all $k\in\NN\cup \{0\}$ the set of all $f\in C^{k}(\RR^2)$ with $D^{\alpha}f(\cdot+z)=D^{\alpha}f$ for all $z\in\ZZ^2$ and all multiindices $\alpha$ of order up to $k$ is denoted by $C^{k}(\mathcal{Y})$. $C^k$ functions with compact support are denoted by a subindex $0$. For any open set $A$, we denote by $L^2(\mathcal{Y})$, $W^{1,2}(\mathcal{Y})$ and $W^{1,2}(A\times\mathcal{Y})$ the Banach spaces obtained as closures of $C^{\infty}(\mathcal{Y})$ and $C^{\infty}(\overline{A}, C^{\infty}(\mathcal{Y}))$ with respect to the norm in $L^2(Y)$, $W^{1,2}(Y)$ and $W^{1,2}(A\times Y)$, respectively. An additional dot (e.g. in $\dot{L}^2(\mathcal{Y})$) denotes functions with average zero over $\mathcal{Y}$.

\vspace{0.5 cm}
\subsubsection*{Geometry of a periodically wrinkled shell}\label{sec:31}
\vspace{0.5 cm}

Let $\Omega$ be a two-dimensional domain with the boundary $\partial\Omega$ of the class $C^3$ in the plane spanned by $\vect{e_1},\vect{e_2}$ and by $x$ we denote $x=(x_1,x_2)$ a generic point in $\Omega$. By $\Omega$ we denote $\Omega^1$. The canonical cell in $\RR^2$ we donote by $Y=[0,1)^2$, we consider the generic point in $Y$ is $y=(y_1,y_2)$ and by  $x_{\varepsilon}=x/\varepsilon$. By a periodically wrinkled shell we mean a shell defined in the following way. Let $\theta:\RR^2\to\RR$ be an $Y-$periodic function of class $C^3$. We call $\theta$ the shape function.  Let $\psi: \RR^2\to\RR^3$ be a function of class $C^3$. We consider a three-dimensional periodically wrinkled shell with wrinkles of scale $\varepsilon $ occupying in its reference configuration the set $\overline{\Omega^{\varepsilon,p}}$, where ${\Omega}^{\varepsilon, p}=\vect{\Theta}^{\varepsilon, p}(\Omega);$ the mapping $\vect{\Theta}^{\varepsilon, p}:\{\Omega\}\to \RR^3$ is given by
\begin{equation*}
\vect{\Theta}^{\varepsilon, p}(x)=\displaystyle\psi(x)+\varepsilon^p\theta\left(x_{\varepsilon}\right)\vect{a}_{3}(x),
\end{equation*}
for all $x\in\overline{\Omega}$, where $\vect{a}_3$ is a unit normal vector to the graph of $\psi$. At each point  $x$ of $\overline{\Omega}$ the vector $\vect{a}_3$ is given by
\begin{equation}\label{eq:0}
\vect{a}_3(x)=\frac{\vect{a}_1(x)\wedge \vect{a}_2(x)}{|\vect{a}_1(x)\wedge \vect{a}_2(x)|}=\frac{\vect{a}_1(x)\wedge \vect{a}_2(x)}{\sqrt{a(x)}},\quad\text{where,}\,\vect{a}_{\alpha}(x)=\partial_{\alpha} \psi(x),\quad\alpha=1,2.
\end{equation}

\vspace{0.5 cm}
We are interested in the case with $p=2$, i.e., $\Omega^{\varepsilon}:=\vect{\Theta}^{\varepsilon}(\Omega)$, where, $\vect{\Theta}^{\varepsilon}(x):=\vect{\Theta}^{\varepsilon,2}(x)$ which is called moderately wrinkled shell.

\vspace{0.5 cm}

\subsection{The Koiter's shell model}\label{sec:32}

\vspace{1.0 cm}

\begin{equation}\label{eq:alph2}
\vect{a}^{\varepsilon}_{\alpha}(x):=\partial_{\alpha}\Theta^{\varepsilon}(x)=\vect{a}_{\alpha}(x)+\varepsilon\partial_{\alpha}\theta (x_{\varepsilon})\vect{a}_3(x)+\varepsilon^2 \theta (x_{\varepsilon})\partial_{\alpha}\vect{a}_3(x)
\end{equation}
for $\alpha\in\{1,2\}$.
The next step is to find a explicit form for $\vect{a}_3^{\varepsilon}(x)$. For this end, we define
\begin{equation*}
\vect{a}_3^{\varepsilon}(x)\nonumber:=\displaystyle\frac{\vect{a}_1^{\varepsilon}(x)\wedge \vect{a}_2^{\varepsilon}(x)}{|\vect{a}_1^{\varepsilon}(x)\wedge \vect{a}_2^{\varepsilon}(x) |},
\end{equation*}
 then we obtain
\begin{align}
\vect{a}_3^{\varepsilon}(x)&\nonumber=\displaystyle\frac{1}{\sqrt{a_{\varepsilon}(x)}}\left[(\vect{a}_1\wedge \vect{a}_2)(x)+\varepsilon\partial_{2}\theta (x_{\varepsilon}) (\vect{a}_1\wedge \vect{a}_3)(x)+\varepsilon^2 \theta (x_{\varepsilon})(\vect{a}_1\wedge \partial_2 \vect{a}_3)(x) \right.\\
&\nonumber\quad+\varepsilon\partial_{1}\theta (x_{\varepsilon})(\vect{a}_3\wedge \vect{a}_2)(x)+\varepsilon^3\theta(x_{\varepsilon})\partial_{1}\theta (x_{\varepsilon})(\vect{a}_3\wedge \partial_{2} \vect{a}_3)(x)+\varepsilon^2 \theta(x_{\varepsilon})(\partial_1\vect{a}_3\wedge \vect{a}_2)(x)\\
&\label{eq:5}\left.\quad+\varepsilon^3\theta(x_{\varepsilon})\partial_{2}\theta(x_{\varepsilon})(\partial_{1}\vect{a}_3\wedge \vect{a}_3)(x)+\varepsilon^4\theta^2(x_{\varepsilon})(\partial_{1}\vect{a}_3\wedge \partial_{2}\vect{a}_3)(x)\right],
\end{align}
where $\sqrt{a_{\varepsilon}(x)}:=|\vect{a}_1^{\varepsilon}(x)\wedge \vect{a}_2^{\varepsilon}(x) |$.

\vspace{0.5 cm}

Let be $\vect{k}_{ij}(x)=(\vect{a}_{i}\wedge \vect{a}_{j})(x)$, $\vect{\ell}_{i\alpha}(x)=(\vect{a}_{i}\wedge \partial_{\alpha}\vect{a}_3)(x)$,  $c_{\alpha}(x_{\varepsilon})=\theta(x_{\varepsilon})\partial_{\alpha}\theta (x_{\varepsilon})$ for $i,j\in\{1,2,3\}$, $\alpha\in\{1, 2\}$ and $\vect{m}(x)=(\partial_{1}\vect{a}_3\wedge\partial_{2}\vect{a}_3)(x)$, then (\ref{eq:5}) can be rewritten as follow
\begin{align}
\vect{a}_3^{\varepsilon}(x)&\nonumber=\displaystyle\frac{1}{\sqrt{a_{\varepsilon}(x)}}\left[\vect{k}_{12}(x)+\varepsilon\left(\partial_{2}\theta (x_{\varepsilon})\vect{k}_{13}(x)+\partial_{1}\theta (x_{\varepsilon})\vect{k}_{32}(x)\right)  \right.\\
&\left.\label{eq:6}\quad+\varepsilon^3\left(c_1(x_{\varepsilon})\vect{\ell}_{32}(x)-c_2(x_{\varepsilon})\vect{\ell}_{31}(x) \right)+\varepsilon^2\theta(x_{\varepsilon})\left(\vect{\ell}_{12}(x)-\vect{\ell}_{21}(x)\right)+\varepsilon^4\theta^2(x_{\varepsilon})\vect{m}(x)\right],
\end{align}

therefore
\begin{equation}\label{eq:7}
\vect{a}_3^{\varepsilon}(x)=\displaystyle\frac{1}{\sqrt{a_{\varepsilon}(x)}}\left[\vect{k}_{12}(x)+\varepsilon\left(\partial_{2}\theta (x_{\varepsilon})\vect{k}_{13}(x)+\partial_{1}\theta (x_{\varepsilon})\vect{k}_{32}(x)\right)+ \mathcal{O}(\varepsilon^2)\right].
\end{equation}
In the other hand, since $\vect{k}_{12}$ are orthogonal with $\vect{k}_{13}$ and $\vect{k}_{32}$, we have that
\begin{equation*}
a_{\varepsilon}(x)=|\vect{k}_{12}|^2+\varepsilon^2|\partial_{2}\theta (x_{\varepsilon})\vect{k}_{13}(x)+\partial_{1}\theta (x_{\varepsilon})\vect{k}_{32}(x)|^2+\mathcal{O}(\varepsilon^3).
\end{equation*}
By Taylor expansion from equation above, we obtain
\begin{equation}\label{eq:8}
\displaystyle\frac{1}{\sqrt{a_\varepsilon(x)}}=\displaystyle\frac{1}{\sqrt{a(x)}}-\frac{1}{2}\frac{\left|\partial_{2}\theta (x_{\varepsilon})\vect{k}_{13}(x)+\partial_{1}\theta (x_{\varepsilon})\vect{k}_{32}(x)\right|^2}{\sqrt{a^3(x)}}\varepsilon^2+\mathcal{O}(\varepsilon^3)
\end{equation}
and
\begin{equation}\label{eq:8***}
\sqrt{a_{\varepsilon}(x)}=\displaystyle\sqrt{a(x)}+\frac{1}{2}\frac{\left|\partial_{2}\theta (x_{\varepsilon})\vect{k}_{13}(x)+\partial_{1}\theta (x_{\varepsilon})\vect{k}_{32}(x)\right|^2}{\sqrt{a^3(x)}}\varepsilon^2+\mathcal{O}(\varepsilon^3).
\end{equation}
Finally, from (\ref{eq:7}), (\ref{eq:8}) and  (\ref{eq:0}) we deduce
\begin{align}
\vect{a}_3^{\varepsilon}(x)&\nonumber=\vect{a}_3(x)+\displaystyle\frac{\left(\partial_{2}\theta (x_{\varepsilon})\vect{k}_{13}(x)+\partial_{1}\theta (x_{\varepsilon})\vect{k}_{32}(x)\right)}{\sqrt{a(x)}}\varepsilon-\displaystyle\frac{1}{2}\frac{\left|\partial_{2}\theta (x_{\varepsilon})\vect{k}_{13}(x)+\partial_{1}\theta (x_{\varepsilon})\vect{k}_{32}(x)\right|^2}{a(x)}\vect{a}_{3}\varepsilon^2\\
&\label{eq:9}-\displaystyle\frac{1}{2}\frac{\left|\partial_{2}\theta (x_{\varepsilon})\vect{k}_{13}(x)+\partial_{1}\theta (x_{\varepsilon})\vect{k}_{32}(x)\right|^2}{a(x)\sqrt{a(x)}}\left(\partial_{2}\theta (x_{\varepsilon})\vect{k}_{13}(x)+\partial_{1}\theta (x_{\varepsilon})\vect{k}_{32}(x)\right)\varepsilon^3+\mathcal{O}(\varepsilon^4).
\end{align}
Now,  for $\alpha\in\{1,2\}$ we obtain
\begin{align}
\partial_{\alpha}\vect{a}_3^{\varepsilon}(x)&\nonumber=\partial_{\alpha}\vect{a}_3(x)+\displaystyle\frac{1}{\sqrt{a(x)}}\left[\partial_{\alpha 2}\theta (x_{\varepsilon})\vect{k}_{13}(x)+\partial_{\alpha 1}\theta (x_{\varepsilon})\vect{k}_{32}(x)\right]\\
&\quad+\frac{1}{\sqrt{a(x)}}\left[\nonumber\partial_{2}\theta (x_{\varepsilon})\partial_{\alpha}\vect{k}_{13}(x)+\partial_{1}\theta (x_{\varepsilon})\partial_{\alpha}\vect{k}_{32}(x)\right]\varepsilon\\
&\label{eq:10}\quad-\displaystyle\frac{\partial_{\alpha}a(x)}{2a(x)\sqrt{a(x)}}\left[\partial_{2}\theta (x_{\varepsilon})\vect{k}_{13}(x)+\partial_{1}\theta (x_{\varepsilon})\vect{k}_{32}(x)\right]\varepsilon+\mathcal{O}(\varepsilon^2).
\end{align}
From (\ref{eq:alph2}) we conclude that for $\alpha=1,2$
\begin{equation*}
\vect{a}_{\alpha}^{\varepsilon}(x)=\vect{a}_{\alpha}(x)+\varepsilon\partial_{\alpha}\theta (x_{\varepsilon})\vect{a}_3(x)-\varepsilon^2\theta (x_{\varepsilon})b_{\alpha\sigma}\vect{a}^{\sigma}(x)
\end{equation*}
and
\begin{align}
\vect{a}_3^{\varepsilon}(x)&\nonumber=\vect{a}_3(x)+\displaystyle\frac{\left(\partial_{2}\theta (x_{\varepsilon})\vect{k}_{13}(x)+\partial_{1}\theta (x_{\varepsilon})\vect{k}_{32}(x)\right)}{\sqrt{a(x)}}\varepsilon-\displaystyle\frac{1}{2}\frac{\left|\partial_{2}\theta (x_{\varepsilon})\vect{k}_{13}(x)+\partial_{1}\theta (x_{\varepsilon})\vect{k}_{32}(x)\right|^2}{a(x)}\vect{a}_{3}(x)\varepsilon^2\\
&\label{eq:12}-\displaystyle\frac{1}{2}\frac{\left|\partial_{2}\theta (x_{\varepsilon})\vect{k}_{13}(x)+\partial_{1}\theta (x_{\varepsilon})\vect{k}_{32}(x)\right|^2}{a(x)\sqrt{a(x)}}\left(\partial_{2}\theta (x_{\varepsilon})\vect{k}_{13}(x)+\partial_{1}\theta (x_{\varepsilon})\vect{k}_{32}(x)\right)\varepsilon^3+\mathcal{O}(\varepsilon^4).
\end{align}

We denote,
\begin{equation}\label{eq:13}
\sqrt{a_{\varepsilon}(x)}=\det \nabla \vect{\Theta}^{\varepsilon}(x)= \vect{a}_{1}^{\varepsilon}(x)\cdot\left(\vect{a}_{2}^{\varepsilon}\wedge \vect{a}_{3}^{\varepsilon}\right)(x),
\end{equation}
from (\ref{eq:alph2}) and (\ref{eq:12}) we have,
\begin{align*}
\sqrt{a_{\varepsilon}(x)}&= \vect{a}_{1}^{\varepsilon}(x)\cdot\left(\vect{a}_{2}^{\varepsilon}\wedge \vect{a}_{3}^{\varepsilon}\right)(x)=\left(\vect{a}_{1}^{\varepsilon}\wedge \vect{a}_{2}^{\varepsilon}\right)(x)\cdot \vect{a}_{3}^{\varepsilon}(x)\\
&=\left\{\left[\vect{a}_{1}(x)+\varepsilon\partial_{1}\theta (x_{\varepsilon})\vect{a}_3(x)+\mathcal{O}(\varepsilon^2)\right]\wedge \left[\vect{a}_{2}(x)+\varepsilon\partial_{2}\theta (x_{\varepsilon})\vect{a}_3(x)+\mathcal{O}(\varepsilon^2)\right]\right\}\cdot \vect{a}_{3}^{\varepsilon}(x)\\
&=\sqrt{a(x)}+\mathcal{O}(\varepsilon^2),
\end{align*}
and from here deduce the last equality of the lemma.
For the contravariant form of $[\vect{a}^{\varepsilon}_{i}(x)]$, denoted by $[{}^{\varepsilon}\vect{a}^{i}(x)]$ which is defined as follows
\begin{equation}\label{countervariant1}
{}^{\varepsilon}\vect{a}^{i}(x)\cdot \vect{a}^{\varepsilon}_{j}(x)=\delta_{ij},\quad x\in \Omega,\, i,j=1,2,3;
\end{equation} 
we have
\begin{equation}\label{eq:14**}
\left(\nabla \vect{\Theta}^{\varepsilon}\right)^{-1}(x)=(\vect{a}^{\varepsilon}_{1}|\vect{a}^{\varepsilon}_2|\vect{a}^{\varepsilon}_3)^{-1}(x)=({}^{\varepsilon}\vect{a}^{1}|{}^{\varepsilon}\vect{a}^{2}|{}^{\varepsilon}\vect{a}^{3})(x)=\frac{1}{\sqrt{a_{\varepsilon}}}\begin{pmatrix}
      (\vect{a}_{2}^{\varepsilon}\wedge \vect{a}_{3}^{\varepsilon})^{T} &  \\
      (\vect{a}_{3}^{\varepsilon}\wedge \vect{a}_{1}^{\varepsilon})^{T}  &  \\
      (\vect{a}_{1}^{\varepsilon}\wedge \vect{a}_{2}^{\varepsilon})^{T} & 
   \end{pmatrix}(x).
\end{equation}
 For $i=1,2,3$, we need compute the values for each ${}^{\varepsilon}\vect{a}^{i}(x)$ to finally for $\alpha, \beta, \varrho=1,2$ find the fundamental forms $(a^{\varepsilon}_{\alpha\beta}), (b^{\varepsilon}_{\alpha\beta}), (c^{\varepsilon}_{\alpha\beta})$  of the manifold $\Omega^{\varepsilon}$ and the Christoffel's symbols ${}^{\varepsilon}\Gamma^{\varrho}_{\alpha\beta}$ which are defined as follows:
\begin{align}
a^{\varepsilon}_{\alpha\beta}(x)&\label{fundform1}=\vect{a}^{\varepsilon}_{\alpha}(x)\cdot\vect{a}^{\varepsilon}_{\beta}(x),\\
b^{\varepsilon}_{\alpha\beta}(x)&\label{fundaform2}=\vect{a}^{\varepsilon}_{3}(x)\cdot\partial_{\beta}\vect{a}^{\varepsilon}_{\alpha}(x),\\
{}^{\varepsilon}\Gamma^{\sigma}_{\alpha\beta}(x)&\label{cristoffell1}={}^{\varepsilon}\vect{a}^{\varrho}(x)\cdot\partial_{\alpha}\vect{a}^{\varepsilon}_{\beta}(x),\\
c^{\varepsilon}_{\alpha\beta}(x)&\label{fundaform3}={}^{\varepsilon}b^{\lambda}_{\alpha}(x)b^{\varepsilon}_{\lambda\beta}(x),
\end{align}
with
\begin{equation*}
{}^{\varepsilon}b^{\beta}_{\alpha}(x)={}^{\varepsilon}a^{\beta\varrho}(x)b^{\varepsilon}_{\varrho\alpha}(x)\quad\text{and}\quad [{}^{\varepsilon}a^{\alpha\beta}(x)]=[a^{\varepsilon}_{\alpha\beta}(x)]^{-1}.
\end{equation*}

To this aim, we can use the fact that for $\alpha=1,2$
\begin{align*}
\vect{a}_{\alpha}\cdot\left(\vect{a}_1\wedge \vect{a}_{3}\right)=-\sqrt{a}\,\delta_{2\alpha},\\
\vect{a}_{\alpha}\cdot\left(\vect{a}_3\wedge \vect{a}_{2}\right)=-\sqrt{a}\,\delta_{1\alpha}.
\end{align*}

After simple calculation we obtain for $\alpha, \beta, \sigma=1,2$ and $i=1,2,3$ that
\begin{align*}
a^{\varepsilon}_{\alpha\beta}&=a_{\alpha\beta}+[\displaystyle \partial_{\alpha}\theta(x_{\varepsilon})\partial_{\beta}\theta(x_{\varepsilon})-\theta(x_{\varepsilon})(b_{\beta\sigma}\delta_{\sigma\alpha}+b_{\alpha\sigma}\delta_{\sigma\beta})]\varepsilon^2+\mathcal{O}(\varepsilon^3)\\
a^{\varepsilon}_{i3}&
=\delta_{i3}+\displaystyle\frac{|\partial_2\theta(x_{\varepsilon})\vect{k}_{13}+\partial_1\theta(x_{\varepsilon})\vect{k}_{32}|^2}{a}\delta_{i3}\,\varepsilon^2+\mathcal{O}(\varepsilon^3),
\end{align*}
rewriting we have that for $\alpha, \beta=1,2$ and $i=1,2,3$ that
\begin{align}
a^{\varepsilon}_{\alpha\beta}&\label{matrix1}=a_{\alpha\beta}+[\displaystyle \partial_{\alpha}\theta(x_{\varepsilon})\partial_{\beta}\theta(x_{\varepsilon})-2\theta(x_{\varepsilon})b_{\alpha\beta}]\varepsilon^2+\mathcal{O}(\varepsilon^3)\\
a^{\varepsilon}_{i3}&\label{matrix2}
=\delta_{i3}+\displaystyle\frac{|\partial_2\theta(x_{\varepsilon})\vect{k}_{13}+\partial_1\theta(x_{\varepsilon})\vect{k}_{32}|^2}{a}\delta_{i3}\,\varepsilon^2+\mathcal{O}(\varepsilon^3).
\end{align}
From (\ref{matrix1}) and (\ref{matrix2}) we deduce that
\begin{equation*}
(a^{\varepsilon}_{ij})=\textbf{G}+\displaystyle \varepsilon^2\,\textbf{H}+\mathcal{O}(\varepsilon^3),
\end{equation*}
therefore
\begin{equation*}
({}^{\varepsilon}a^{ij})=(a^{\varepsilon}_{ij})^{-1}=\textbf{G}^{-1}-\displaystyle \varepsilon^2\,\textbf{G}^{-1}\textbf{H}\textbf{G}^{-1}+\mathcal{O}(\varepsilon^3).
\end{equation*}
From the above equation we conclude that
\begin{align}
{}^{\varepsilon}a^{\alpha\beta}&\nonumber=\textbf{G}^{-1}_{\alpha\beta}-\displaystyle \varepsilon^2\,\textbf{G}^{-1}_{\alpha\beta}\textbf{H}_{\alpha\beta}\textbf{G}^{-1}_{\alpha\beta}+\mathcal{O}(\varepsilon^3)\\
&\label{eq:1687}=a^{\alpha\beta}+\left(2a^{\alpha\sigma}b^{\beta}_{\sigma}-a^{\alpha\sigma}a^{\beta\tau}\partial_{\tau}\theta(x_{\varepsilon})\partial_{\sigma}\theta(x_{\varepsilon})\right)\displaystyle \varepsilon^2+\mathcal{O}(\varepsilon^3).
\end{align}
By the same way we can obtain,
\begin{equation}\label{eq:1787}
{}^{\varepsilon} a^{i3}=\delta^{i3}+\displaystyle\frac{|\partial_2\theta(x_{\varepsilon}) \vect{k}_{13}+\partial_{1}\theta(x_{\varepsilon}) \vect{k}_{32}|^2}{a}\delta^{i3}\varepsilon^2+\mathcal{O}(\varepsilon^3).
\end{equation}
Now, from (\ref{eq:1687}) and (\ref{eq:1787}) we compute 
\begin{align}
{}^{\varepsilon}\vect{a}^{\alpha}&\label{eq:1887}
={}^{\varepsilon}\vect{a}^{\alpha i}\cdot \vect{a}^{\varepsilon}_{i}=\vect{a}^{\alpha}+\left(b^{\alpha}_{\sigma}\vect{a}^{\sigma}-a^{\alpha\sigma}\partial_{\sigma}\theta(x_{\varepsilon})\partial_{\tau}\theta(x_{\varepsilon})\vect{a}^{\tau}\right) \varepsilon^2+\mathcal{O}(\varepsilon^3),\\
{}^{\varepsilon}\vect{a}^3&\nonumber=\vect{a}^3+\displaystyle\frac{\left(\partial_2\theta(x_{\varepsilon}) \vect{k}_{13}+\partial_{1}\theta(x_{\varepsilon}) \vect{k}_{32}\right)}{\sqrt{a}}\varepsilon+\displaystyle\frac{|\partial_2\theta(x_{\varepsilon}) \vect{k}_{13}+\partial_{1}\theta(x_{\varepsilon}) \vect{k}_{32}|^2}{a}\vect{a}^3 \varepsilon^2\\
&\label{eq:1987}\quad+\displaystyle\frac{|\partial_2\theta(x_{\varepsilon}) \vect{k}_{13}+\partial_{1}\theta(x_{\varepsilon}) \vect{k}_{32}|^2}{a\sqrt{a}}\left(\partial_2\theta(x_{\varepsilon}) \vect{k}_{13}+\partial_{1}\theta(x_{\varepsilon}) \vect{k}_{32}\right)\varepsilon^3+\mathcal{O}(\varepsilon^4).
\end{align}
\lqqd

\newpage

We assume that $\vect{f}_{\varepsilon}(x)=f^{i}_{\varepsilon}(x)\vect{e}_i:\Omega\to\RR^3$ is a function describing a force density acting on the middle surface $\Omega^{\varepsilon}$ of the corrugated shell; we assume that $\vect{f}_{\varepsilon}\in L^2(\Omega)^3$. We consider the simplest boundary conditions, i.e., we assume that the corrugated shell is clamped on its lateral boundary. The the equilibrium displacement $\vect{u}^{\varepsilon}(x)=u^{\varepsilon}_{i}(x)\vect{a}^{i}_{\varepsilon}(x)$ of the point $\Theta^{\varepsilon, p}(x)\in\Omega^{\varepsilon}$ satisfies:
\begin{equation*}
u^{\varepsilon}_{i}(x)=\partial_{\alpha}u^{\varepsilon}_{3}(x)=0,\quad x\in\partial\Omega,\, i=1,2,3,\, \alpha=1,2.
\end{equation*}

\vspace{0.5 cm}

The starting point of our analysis is the three dimensional linear elasticity problem for the periodically wrinkled shell. The unknow $\vect{u}^{\varepsilon}=({u}_{i}^{\varepsilon})$ satisfies the following variational problem in cartesian coordinates $\mathcal{P}(\Omega^{\varepsilon})$: Find $\vect{u}^{\varepsilon}\in \vect{H}:=H^1_0(\Omega)\times H^1_0(\Omega)\times H^2_0(\Omega)$
 \begin{equation}\label{cartesian}
 B^{\varepsilon}(\vect{u}^{\varepsilon}, \vect{v})=L^{\varepsilon}(\vect{v}),\quad \forall \vect{v}\in\vect{H},
 \end{equation}
 where
\begin{equation}\label{cartesian1}
L^{\varepsilon}(\vect{v})=\int_{\Omega}\sqrt{a_{\varepsilon}} \vect{f}_{\varepsilon}\cdot \vect{v}\,dx,
\end{equation}
and for $\vect{u}, \vect{v}$ in $\vect{H}$,
\begin{equation}\label{cartesian2}
B^{\varepsilon}(\vect{u}, \vect{v})=\displaystyle\int_{\Omega}b_{\varepsilon}^{\alpha\beta\varrho\sigma}\left[d\gamma_{\varrho\sigma}^{\varepsilon}(\vect{u})\gamma_{\alpha\beta}^{\varepsilon}(\vect{v})+\displaystyle\frac{d^3}{3}\Gamma_{\varrho\sigma}^{\varepsilon}(\vect{u})\Gamma_{\alpha\beta}^{\varepsilon}(\vect{v})\right]\sqrt{a_{\varepsilon}}\,dx.
\end{equation}
Here $2d$ is the thickness of the shell, and for $\alpha, \beta,\varrho, \sigma=1,2$ and $x\in\Omega$,
\begin{equation}\label{tensor}
b_{\varepsilon}^{\alpha\beta\varrho\sigma}=\displaystyle\frac{4\lambda\mu}{\lambda+2\mu}a^{\alpha\beta}_{\varepsilon}(x)a^{\varrho\sigma}_{\varepsilon}(x)+2\mu\left[a^{\alpha\varrho}_{\varepsilon}(x)a^{\beta\sigma}_{\varepsilon}(x)+a^{\alpha\sigma}_{\varepsilon}(x)a^{\beta\varrho}_{\varepsilon}(x)\right],
\end{equation}
where $\lambda>0$ and $\mu>0$ are the Lame constants of the material, and for $\vect{u}$ in $\vect{H}$ and $\alpha, \beta,\varrho, \sigma=1,2$,
\begin{align}
\gamma^{\varepsilon}_{\alpha\beta}(\vect{u})&\label{tensor1}=e_{\alpha\beta}(\vect{u})-{}^{\varepsilon}\Gamma^{\varrho}_{\alpha\beta}u_{\varrho}-b^{\varepsilon}_{\alpha\beta}u_3,\quad\text{with}\,\, e_{\alpha\beta}(\vect{u})=\displaystyle\frac{1}{2}\left(\partial_{\alpha}u_{\beta}+\partial_{\beta}u_{\alpha}\right),\\
\Gamma^{\varepsilon}_{\alpha\beta}(\vect{u})&\label{tensor2}=\partial_{\alpha\beta}u_3-{}^{\varepsilon}\Gamma_{\alpha\beta}^{\varrho}\partial_{\varrho}u_3+{}^{\varepsilon}b^{\varrho}_{\beta}(\partial_{\alpha}u_{\varrho}-{}^{\varepsilon}\Gamma^{\sigma}_{\varrho\alpha}u_{\sigma})-c^{\varepsilon}_{\alpha\beta}u_3\\
&\nonumber\quad+{}^{\varepsilon}b^{\varrho}_{\alpha}(\partial_{\beta}u_{\varrho}-{}^{\varepsilon}\Gamma^{\sigma}_{\varrho\beta}u_{\sigma})+(\partial_{\alpha}{}^{\varepsilon}b^{\varrho}_{\beta}+{}^{\varepsilon}\Gamma^{\varrho}_{\alpha\sigma}{}^{\varepsilon}b^{\sigma}_{\beta}-{}^{\varepsilon}\Gamma^{\sigma}_{\alpha\beta}{}^{\varepsilon}b^{\varrho}_{\sigma})u_{\varrho}.
\end{align}
We are interested in to find the problem limit of (\ref{cartesian}) as $\varepsilon\to 0$ in a suitable sense.

\vspace{1.0 cm}

\begin{teo}\label{teo:332}
The functions $\gamma^{\varepsilon}_{\alpha\beta}(\vect{u})$ and $\Gamma^{\varepsilon}_{\alpha\beta}(\vect{u})$  defined in (\ref{tensor1}) and (\ref{tensor2}) are of the form
\begin{align}
\gamma^{\varepsilon}_{\alpha\beta}(\vect{u})&\label{anouncedeq1}=\gamma_{\alpha\beta}(\vect{u})-\partial_{\alpha\beta}\theta (x_{\varepsilon})u_3+\varepsilon P^{\varepsilon}_{\alpha\beta}(\vect{u}),\\
\Gamma^{\varepsilon}_{\alpha\beta}(\vect{u})&\label{anouncedeq2}=\Gamma_{\alpha\beta}(\vect{u})+\varepsilon^{-1}a^{\varrho\lambda}\partial_{\alpha\lambda\beta}\theta(x_{\varepsilon})u_{\varrho}+Q^{\varepsilon}_{\alpha\beta}(\vect{u})+\varepsilon R^{\varepsilon} _{\alpha\beta}(\vect{u})
\end{align}
and there exists constant $C$ such that for $\alpha, \beta=1,2$ and for all $\varepsilon>0$,
\begin{equation}\label{boundresidualstraintensor}
|P^{\varepsilon} _{\alpha\beta}(\vect{u})|_{L^2(\Omega)}\leq C |\vect{u}|_{L^2(\Omega)^3}\quad|R^{\varepsilon} _{\alpha\beta}(\vect{u})|_{L^2(\Omega)}\leq C |\vect{u}|_{H^1(\Omega)^3},\quad\text{for all}\,\,\vect{u}\in \vect{H},
\end{equation} 
 and the auxiliary functions $\gamma_{\alpha\beta}(\vect{u}), \Gamma_{\alpha\beta}(\vect{u})$ and $Q^{\varepsilon}_{\alpha\beta}(\vect{u})$ is defined as 
 \begin{align}
 \gamma_{\alpha\beta}(\vect{u})&\label{tensorold1}=e_{\alpha\beta}(\vect{u})-\Gamma^{\varrho}_{\alpha\beta}u_{\varrho}-b
 _{\alpha\beta}u_3,\\
 \Gamma_{\alpha\beta}(\vect{u})&\nonumber=\partial_{\alpha\beta}u_3-\Gamma_{\alpha\beta}^{\varrho}\partial_{\varrho}u_3+b^{\varrho}_{\beta}(\partial_{\alpha}u_{\varrho}-\Gamma^{\sigma}_{\varrho\alpha}u_{\sigma})-c_{\alpha\beta}u_3\\
&\label{tensorold2}\quad+b^{\varrho}_{\alpha}(\partial_{\beta}u_{\varrho}-\Gamma^{\sigma}_{\varrho\beta}u_{\sigma})+(\partial_{\alpha}b^{\varrho}_{\beta}+\Gamma^{\varrho}_{\alpha\sigma}b^{\sigma}_{\beta}-\Gamma^{\sigma}_{\alpha\beta}b^{\varrho}_{\sigma})u_{\varrho},\\
 Q^{\varepsilon}_{\alpha\beta}(\vect{u})&\nonumber=a^{\varrho\lambda}[\partial_{\lambda\beta}\theta(x_{\varepsilon})\partial_{\alpha}u_{\varrho}+\partial_{\lambda\alpha}\theta(x_{\varepsilon})\partial_{\beta}u_{\varrho}]-a^{\lambda\varrho}\partial_{\varrho\alpha}\theta(x_{\varepsilon})\partial_{\lambda\beta}\theta(x_{\varepsilon})u_3\\
 &\nonumber-[b^{\lambda}_{\alpha}\partial_{\lambda\beta}\theta(x_{\varepsilon})+b^{\varrho}_{\beta}\partial_{\varrho\alpha}\theta(x_{\varepsilon})]u_3+\partial_{\alpha}a^{\varrho\lambda}\partial_{\lambda\beta}\theta(x_{\varepsilon})u_{\varrho}\\
 &\nonumber-a^{\varrho\lambda}[\partial_{\lambda\sigma}\theta(x_{\varepsilon})\Gamma^{\sigma}_{\alpha\beta}u_{\varrho}+\partial_{\lambda\alpha}\theta(x_{\varepsilon})\Gamma^{\sigma}_{\varrho\beta}u_{\sigma}]\\
 &+\label{tensornew2021}\displaystyle\frac{\left(\partial_{\alpha2}\theta(x_{\varepsilon}) \vect{k}_{13}+\partial_{\alpha1}\theta(x_{\varepsilon}) \vect{k}_{32}\right)}{\sqrt{a}}a^{\varrho\lambda}\partial_{\lambda}a_{\beta}u_{\varrho}.
 \end{align}
 
\end{teo}

\vspace{0.5 cm}

\dem
The results (\ref{anouncedeq1})-(\ref{anouncedeq2}) follows from (\ref{tensor1})-(\ref{tensor2}), (\ref{fundform1})-(\ref{fundaform3}) and the expansions given in (\ref{eq:alph2}),(\ref{eq:12}), (\ref{eq:1887}) and (\ref{eq:1987}).

\lqqd

\vspace{1.0 cm}

\vspace{0.5 cm}

\vspace{1.0 cm}

\section{Moderately wrinkled shell}

\vspace{0.5 cm}

\subsection{A priori estimates}

\vspace{0.5 cm}

Let 
\begin{align*}
||\vect{v}||_{\varepsilon}&=\left[|\vect{v}|^2_{H^1(\Omega)^3}+\displaystyle\sum_{\alpha,\beta=1}^{2}\left|\partial_{\alpha\beta}v_3+\varepsilon^{-1}a^{\varrho\lambda}\partial_{\alpha\lambda\beta}\theta(x_{\varepsilon})v_{\varrho}\right|^2_{L^2(\Omega)}\right]^{1/2}.
\end{align*}

\vspace{1.0 cm}

\begin{teo}\label{teo1:aprioriestimate1}
There exists a constant $C>0$ such that 
\begin{equation}\label{bound1}
B^{\varepsilon}(\vect{v}, \vect{v}) \geq C||\vect{v}||^2_{\varepsilon},\quad \forall \vect{v}\in \widetilde{\vect{H}}=H^1_0(\Omega)^2\times [H^2_0(\Omega)\cap H^1_0(\Omega)].
\end{equation}
\end{teo}

\vspace{0.5 cm}

\dem 
Assume that (\ref{bound1}) does not hold. Then there exists sequences $\left\{\varepsilon_n\right\}_{n\in\NN}$ in $\RR^{+}$ and $\left\{\vect{v}^n\right\}_{n\in \NN}$ in $\widetilde{\vect{H}}$. such that
\begin{align}
&\label{eq:8701}\varepsilon_n\in(0,1),\,\varepsilon_n\rightarrow 0\quad\text{and}\quad B^{\varepsilon_n}(\vect{v}_n, \vect{v}_n) \rightarrow 0\quad\text{as}\quad n\rightarrow \infty,\\
&\label{eq:8702}||\vect{v}^n||_{\varepsilon_n}=1,\quad \forall n\in\NN.
\end{align}
Using \cite[Theorem 2.1]{REF2}, we obtain from (\ref{eq:8701}) that
\begin{equation}\label{eq:8703}
\gamma_{\alpha\beta}^{\varepsilon_n}(\vect{v}^n), \Gamma_{\alpha \beta}^{\varepsilon_n}(\vect{v}^n)\rightarrow 0\quad\text{in}\quad L^2(\Omega)\quad\text{as}\quad n\rightarrow \infty,\, \alpha, \beta=1,2.
\end{equation}

For $n\in\NN$ and $\alpha=1,2$, we define the auxiliary functions 
\begin{align}
&\label{eq:8704}
\begin{cases}
w^n_{\alpha}=v^n_{\alpha}-\varepsilon_n\partial_{\alpha}\theta(x/\varepsilon_n)v^n_3, &\\
w^n_{3}=v^n_3,
\end{cases}\\
&\label{eq:8705}
\begin{cases}
u^n_{\alpha}=\partial_{\alpha}v^n_{3}+\partial_{\alpha\varrho}\theta(x/\varepsilon_n)a^{\varrho \sigma}v^n_{\sigma}, &\\
u^n_{3}=v^n_{3}
\end{cases}
\end{align}

Let $\vect{w}^n=(w^n_i)$, $\vect{u}^n=(u^n_i)$. From (\ref{tensorold1}) and (\ref{eq:8704}) we see that
\begin{equation}\label{eq:8706}
\gamma_{\alpha\beta}^{\varepsilon_n}(\vect{v}^n)=\gamma_{\alpha\beta}(\vect{w}^n)+\mathcal{O}(\varepsilon_n),
\end{equation}
and using Korn's inequatily and (\ref{eq:8703}) we obtain
\begin{equation}\label{eq:8707}
|w^n_{\alpha}|_{H^1(\Omega)}, |w^n_3|_{L^2(\Omega)}\rightarrow 0\quad\text{as}\quad n\to\infty, \alpha=1,2.
\end{equation}
Analogously,  from (\ref{tensorold2}) and (\ref{eq:8704})-(\ref{eq:8705}) we see that
\begin{align}
\Gamma_{\alpha \beta}^{\varepsilon_n}(\vect{v}^n)&\nonumber=\gamma_{\alpha\beta}(\vect{u}^n)+\displaystyle\frac{1}{2}a^{\varrho\lambda}[\partial_{\lambda\beta}\theta(x_{\varepsilon})\partial_{\alpha}w^n_{\varrho}+\partial_{\lambda\alpha}\theta(x_{\varepsilon})\partial_{\beta}w^n_{\varrho}]-\displaystyle\frac{1}{2}\partial_{\alpha\varrho}\theta(x_{\varepsilon})\partial_{\beta}a^{\varrho\sigma}w^n_{\sigma}\\
&\nonumber\quad+(\partial_{\alpha}b^{\varrho}_{\beta}+\Gamma^{\varrho}_{\alpha\sigma}b^{\sigma}_{\beta}-\Gamma^{\sigma}_{\alpha\beta}b^{\varrho}_{\sigma})w^n_{\varrho}+b^{\varrho}_{\beta}(\partial_{\alpha}w^n_{\varrho}-\Gamma^{\sigma}_{\varrho\alpha}w^n_{\sigma})+[b_{\alpha\beta}-c_{\alpha\beta}]w^n_3\\
&\nonumber\quad+b^{\varrho}_{\alpha}(\partial_{\beta}w^n_{\varrho}-\Gamma^{\sigma}_{\varrho\beta}w^n_{\sigma})-a^{\varrho\lambda}[\partial_{\lambda\sigma}\theta(x_{\varepsilon})\Gamma^{\sigma}_{\alpha\beta}w^n_{\varrho}+\partial_{\lambda\alpha}\theta(x_{\varepsilon})\Gamma^{\sigma}_{\varrho\beta}w^n_{\sigma}]\\
 &\label{eq:8708}\quad+\displaystyle\frac{\left(\partial_{\alpha2}\theta(x_{\varepsilon}) \vect{k}_{13}+\partial_{\alpha1}\theta(x_{\varepsilon}) \vect{k}_{32}\right)}{\sqrt{a}}a^{\varrho\lambda}\partial_{\lambda}a_{\beta}w^n_{\varrho}+\mathcal{O}(\varepsilon_n),\quad n\in\NN,\,\alpha,\beta=1,2,
\end{align}
and using  (\ref{eq:8703}) and (\ref{eq:8707}) we obtain
\begin{equation}\label{eq:8709}
\gamma_{\alpha\beta}(\vect{u}^n)\rightarrow 0\quad\text{in}\quad L^2(\Omega)\quad\text{as}\quad n\to\infty, \alpha=1,2.
\end{equation}
By (\ref{eq:8702}) the sequence $\{\vect{v}^n\}_{n\in \NN}$ is bounded in $H^1(\Omega)^3$, so there exists $\vect{v}^0$ in $H^1(\Omega)^3$ such that, up to a subsequence,
\begin{equation*}
\vect{v}^n=(v^n_i)\to\vect{v}^0=(v^0_i)\quad\text{weakly in}\,\, H^1(\Omega)^3\,\text{as}\, n\to\infty,
\end{equation*}
and therefore strongly in $L^2(\Omega)^3$. Using this fact we conclude that
\begin{align*}
&
\begin{cases}
w^n_{\alpha}=v^n_{\alpha}-\varepsilon_n\partial_{\alpha}\theta(x/\varepsilon_n)v^n_3\to v^0_{\alpha},\quad\text{in}\,\, L^2(\Omega)\,\text{as}\,\, n\to\infty, \alpha=1,2, &\\
w^n_{3}=v^n_3\to v^0_3, \quad\text{in}\,\, L^2(\Omega)\,\text{as}\,\, n\to\infty,
\end{cases}\\
&
\begin{cases}
u^n_{\alpha}=\partial_{\alpha}v^n_{3}+\partial_{\alpha\varrho}\theta(x/\varepsilon_n)a^{\varrho \sigma}v^n_{\sigma}\to \partial_{\alpha}v^0_{3},\quad\text{weakly in}\,\, L^2(\Omega)\,\text{as}\,\, n\to\infty,  \alpha=1,2,&\\
u^n_{3}=v^n_{3}\to v^0_3, \quad\text{in}\,\, L^2(\Omega)\,\text{as}\,\, n\to\infty.
\end{cases}
\end{align*}
From (\ref{eq:8707}) it follows now that $v^0_i=0$, $i=1,2,3$. The sequences
\begin{equation*}
\partial_{\beta}u^n_{\alpha}=\partial_{\alpha\beta}v^n_3+\frac{1}{\varepsilon_n}a^{\varrho\sigma}\partial_{\alpha\varrho\beta}\theta v^n_{\sigma}+a^{\varrho\sigma}\partial_{\alpha\varrho}\theta \partial_{\beta}v^n_{\sigma}+\partial_{\beta}a^{\varrho\sigma}\partial_{\alpha\varrho}\theta v^n_{\sigma},\quad \alpha,\beta=1,2,
\end{equation*}
are bounded  in $L^2(\Omega)$ because (\ref{eq:8702}), so there exists $\vect{u}^0=(u^0_{\alpha})$ in $H^1(\Omega)^2$ such that, up to a subsequence,
\begin{equation*}
u^n_{\alpha}\rightarrow u^0_{\alpha},\quad\text{weakly in}\,\, H^1(\Omega)\,\text{as}\,\, n\to\infty,  \alpha=1,2.
\end{equation*}
But then $u^0_{\alpha}=\partial_{\alpha}v_3^0$, and consequently $v^0_3$ belongs to $H^2(\Omega)\cap H^1_0(\Omega)$. Moreover, from (\ref{eq:8709}) it follows that $\partial_{\alpha\beta}v_3^0=0$, therefore $v_3^0=\vect{u}^0=0$.

By the definition (\ref{eq:8704}) of $w_{\alpha}^n$ we have
\begin{align*}
\partial_{\beta}w^n_{\alpha}&=\partial_{\beta}v^n_{\alpha}-\partial_{\alpha\beta}\theta v^n_3-\varepsilon_n\partial_{\alpha}\theta\partial_{\beta}v^n_{3},\quad \alpha,\beta=1,2,
\end{align*}
hence (\ref{eq:8707}) implies now
\begin{equation*}
\partial_{\beta}v^n_{\alpha}\rightarrow 0\quad \text{in}\,\, L^2(\Omega)\,\text{as}\,\, n\to\infty, \alpha=1,2,
\end{equation*}
and consequently
\begin{equation}\label{eq:87091}
v^n_{\alpha}\rightarrow 0\quad \text{in}\,\, H^1(\Omega)\,\text{as}\,\, n\to\infty, \alpha=1,2.
\end{equation}
From (\ref{eq:8705}) we find
\begin{align*}
\gamma_{\alpha\beta}(\vect{u}^n)&=\partial_{\alpha\beta}v^n_3+\displaystyle\frac{1}{\varepsilon_n}\partial_{\alpha\beta\varrho}\theta a^{\varrho\sigma}v^n_{\sigma}+\displaystyle\frac{1}{2}\left(\partial_{\alpha\varrho}\theta a^{\varrho\sigma}\partial_{\beta}v^n_{\sigma}+\partial_{\beta\varrho}\theta a^{\varrho\sigma}\partial_{\alpha}v^n_{\sigma}\right)\\
&\quad+\left(\partial_{\alpha\varrho}\theta\partial_{\beta}a^{\varrho\sigma}v^n_{\sigma}+\partial_{\beta\varrho}\theta\partial_{\alpha}a^{\varrho\sigma}v^n_{\sigma}\right)-\Gamma_{\alpha\beta}^{\varrho}\partial_{\varrho}v^n_3-\Gamma_{\alpha\beta}^{\varrho}\partial_{\varrho\lambda}\theta a^{\lambda\sigma}v^n_{\sigma}-b_{\alpha\beta}v^n_3,
\end{align*}
and using (\ref{eq:8709}) and (\ref{eq:87091}) we obtain for $\alpha\beta=1,2$ the following $L^2(\Omega)-$convergence as $n$ tends to infinity:
\begin{align}
\partial_{\alpha\beta}v^n_3+\displaystyle\frac{1}{\varepsilon_n}\partial_{\alpha\beta\varrho}\theta a^{\varrho\sigma}v^n_{\sigma}=&\nonumber\gamma_{\alpha\beta}(\vect{u}^n)-\displaystyle\frac{1}{2}\left(\partial_{\alpha\varrho}\theta a^{\varrho\sigma}\partial_{\beta}v^n_{\sigma}+\partial_{\beta\varrho}\theta a^{\varrho\sigma}\partial_{\alpha}v^n_{\sigma}\right)-\displaystyle\frac{1}{2}\left(\partial_{\alpha\varrho}\theta\partial_{\beta}a^{\varrho\sigma}v^n_{\sigma}+\partial_{\beta\varrho}\theta\partial_{\alpha}a^{\varrho\sigma}v^n_{\sigma}\right)\\
&\label{eq:87092}\quad+\Gamma_{\alpha\beta}^{\varrho}\partial_{\varrho}v^n_3+\Gamma_{\alpha\beta}^{\varrho}\partial_{\varrho\lambda}\theta a^{\lambda\sigma}v^n_{\sigma}-b_{\alpha\beta}v^n_3\to 0.
\end{align}
Finally, from (\ref{eq:87092}) it follows
\begin{equation}\label{eq:87093}
\partial_{\alpha}v^n_3=u^n_{\alpha}-\partial_{\alpha\varrho}\theta a^{\varrho\sigma}v^n_{\sigma}\to 0\quad \text{in}\,\, L^2(\Omega)\,\text{as}\,\, n\to\infty, \alpha=1,2.
\end{equation}

The limits found in (\ref{eq:87091})-(\ref{eq:87093}) are in contradiction with (\ref{eq:8702}).
\lqqd

\vspace{0.5 cm}

We are now able to establish the a priori estimate for the sequence $\{\vect{u}^{\varepsilon}\}_{\varepsilon>0}$ assuming that the right hand sides of (\ref{cartesian}) satisfy
\begin{equation}\label{boundforce}
|\vect{f}_{\varepsilon}|_{L^2(\Omega)^3}\leq C,\quad \forall \varepsilon>0.
\end{equation}

\vspace{0.5 cm}

\begin{teo}\label{aprioriestimate}
(a priori estimate). Assume that (\ref{boundforce}) holds true. Then there exists a constant $C>0$ such that for all $\varepsilon>0$,
\begin{equation}\label{boundnorm}
\begin{cases}
|\vect{u}^{\varepsilon}|_{H^1(\Omega)^3}\leq C, &\\
\left|\partial_{\alpha\beta}u^{\varepsilon}_3+\varepsilon^{-1}a^{\varrho\lambda}\partial_{\alpha\lambda\beta}\theta(x_{\varepsilon})u^{\varepsilon}_{\varrho}\right|_{L^2(\Omega)}\leq C, & \alpha,\beta=1,2.
\end{cases}
\end{equation}
\end{teo}
\vspace{0.5 cm}

\begin{proof}
The proof follows directly from the variational equation (\ref{cartesian}), Theorem \ref{teo1:aprioriestimate1} and assumption (\ref{boundforce}).
\end{proof}

\vspace{1.0 cm}

\subsection{Two-scale limit}

\begin{definition}\label{def:twosc1}
\begin{itemize}
\item[(i)] 
A sequence $(f^{\varepsilon})\subset L^2(\Omega)$ is said to converge weakly two-scale on $\Omega$ to the function $f\in L^2(\Omega\times \mathcal{Y})$ as $\varepsilon\to 0$, provided that the sequence $(f^{\varepsilon})$ is bounded in $L^2(\Omega)$ and
\begin{equation}\label{eq:1****}
\displaystyle\lim_{\varepsilon\to 0}\int_{\Omega}f^{\varepsilon}(x)\varphi\left(x,x_{\varepsilon}\right)dx=\int_{\Omega}\int_{\mathcal{Y}}f(x, y)\varphi\left(x, y\right)dydx,
\end{equation}

for all $\varphi\in C^{0}_{c}\left(\Omega,C^0(\mathcal{Y}) \right)$.

\item[(ii)] We say that $f^{\varepsilon}$ stongly two-scale converges to $f$ if, in addition,
\begin{equation*}
\displaystyle\lim_{\varepsilon\to 0}||f^{\varepsilon}||_{L^2(\Omega)}=||f||_{L^2(\Omega\times \mathcal{Y})}.
\end{equation*}
\end{itemize}
\end{definition}

\medskip We write $f^{\varepsilon}\overset{2}{\rightharpoonup} f$ to denote weak two-scale convergence and $f^{\varepsilon}\overset{2}{\rightarrow} f$ to denote strong two-scale convergence. If $f^{\varepsilon}\overset{2}{\rightharpoonup} f$ then $f^{\varepsilon}\rightharpoonup \int_{\mathcal{Y}}f(\cdot, y)dy$ weakly in $L^2$. If $f^{\varepsilon}$ is bounded in $L^2(\Omega)$ then it has a subsequence which converges weakly three-scale to some $f\in L^2(\Omega; L^2(\mathcal{Y}))$. These and other facts can be deduced from the corresponding results on planar domains (cf. \cite{REF0,REF33}) by means of the following simple observations.

\vspace{0.5 cm}
The following theorem and lemma will be needed later.
\vspace{0.5 cm}

\begin{teo}\label{teo:3.4}
\begin{itemize}
\item[(i)] If $(u_{\varepsilon})_{{\varepsilon}>0}$ is a bounded sequence in $L^2(\Omega)$, then there exists a subsequence, still denoted by  $(u_{\varepsilon})_{{\varepsilon}>0}$, and a function $u^0$ in $L^2(\Omega\times\mathcal{Y})$ such that
\begin{equation*}
u_{\varepsilon}\overset{2}{\rightharpoonup} u^0\quad\text{weakly two-scale as}\,\, {\varepsilon}\to 0.
\end{equation*}
Moreover,
\begin{equation*}
u_{\varepsilon}\rightharpoonup \displaystyle\int_{\mathcal{Y}}u^0\,dy\quad\text{weakly in}\, L^2(\Omega)\,\text{as}\,\, {\varepsilon}\to 0.
\end{equation*}
\item[(ii)] If a sequence $(u_{\varepsilon})_{{\varepsilon}>0}$ converges strongly in $L^2(\Omega)$ to a limit $u^0$ in $L^2(\Omega)$, then
\begin{equation*}
u_{\varepsilon}\overset{2}{\rightharpoonup} u^0\,\,\text{weakly two-scale as}\,{\varepsilon}\to 0.
\end{equation*}
\item[(iii)] Let $(u_{\varepsilon})_{{\varepsilon}>0}$ be a bounded sequence in $H^1(\Omega)$ which converges weakly to a limit $u^0\in H^1(\Omega)$ in the sense of Definition \ref{def:twosc1}. Then $u_{\varepsilon}$ two-scale converges to $u^0(x)$, and there exist a unique function $u^1(x,y)\in L^2(\omega;\dot{H}^1(\mathcal{Y}))$ such that, up to a subsequence, $(\nabla u_{\varepsilon})_{{\varepsilon}>0}\overset{2}{\rightharpoonup}\nabla_{x}u^0+\nabla_{y}u^1(x,y)$.	
\end{itemize}	

\end{teo}

\vspace{0.5 cm}

\begin{lemma}\label{lemma:gamma1}
(cf. \cite[Lemma 3.11]{Velcic1} )Let $(u_{\varepsilon})_{{\varepsilon}>0}$ be a bounded sequence in $L^2(\Omega)$ which weakly two-scale converges to $u^0(x,y)\in L^2(\Omega\times\mathcal{Y})$. Let $(v_{\varepsilon})_{{\varepsilon}>0}$ be a sequence bounded in $L^{\infty}(\Omega)$ which converges in measure to $v_0\in L^{\infty}(\Omega)$. Then $v_{\varepsilon}u_{\varepsilon}\overset{2}{\rightharpoonup} v_0(x) u_0(x,y)$ weakly two-scale in $L^2(\Omega\times\mathcal{Y})$.
\end{lemma}

\vspace{0.5 cm}

The following theorem is essentially like as Theorem 1 in \cite{REF3} and the proof follows from () and \cite[Proposition 1.14]{REF0}.

\vspace{0.5 cm}

\begin{lemma}\label{lemmafundtwoscaleconv1}
There exists a subsequence of $\{\vect{u}^{\varepsilon}\}_{\varepsilon>0}$, still denoted by the same index $\varepsilon$, and functions
\begin{equation*}
\vect{u}^0=(u^0_i)\in H^1_0(\Omega)^3\quad\text{and}\quad\vect{u}^1=(u^1_i)\in L^2(\Omega;\dot{H}^1(\mathcal{Y}))^3
\end{equation*}
such that
\begin{equation*}
\begin{cases}
\vect{u}^{\varepsilon}\rightarrow \vect{u}^0\quad\text{weakly in}\quad H^1(\Omega)^3\quad\text{as}\quad \varepsilon\to 0, &\\
\nabla \vect{u}^{\varepsilon}\rightarrow \nabla \vect{u}^0+\nabla_y \vect{u}^1\quad\text{two-scale as}\quad  \varepsilon\to0. &\\
\end{cases}
\end{equation*}
\end{lemma}

\vspace{0.5 cm}

\begin{lemma}\label{lemma33}
Let $\{\vect{u}^{\varepsilon}\}_{\varepsilon>0}$ be a convergent subsequence as in Lemma \ref{lemmafundtwoscaleconv1}. Then $u^0_3$ belongs to $H^2_0(\Omega)$ and
\begin{equation*}
u^1_3(x,y)=-a^{\varrho\lambda}\partial_{\lambda}\theta(y)u^0_{\varrho}(x).
\end{equation*}
\end{lemma}

\vspace{0.5 cm}

\begin{proof}
\begin{equation}\label{eq:87094}
w^{\varepsilon}_{\alpha}(x)=\partial_{\alpha}u^{\varepsilon}_3(x)+a^{\varrho\lambda}\partial_{\alpha\lambda}\theta(x/\varepsilon)u^{\varepsilon}_{\varrho}(x),\quad \alpha=1,2.
\end{equation}
Then, 
\begin{equation*}
\partial_{\beta}w^{\varepsilon}_{\alpha}(x)=\partial_{\alpha\beta}u^{\varepsilon}_3(x)+\varepsilon_n^{-1}a^{\varrho\lambda}\partial_{\alpha\beta\lambda}\theta(x/\varepsilon)u^{\varepsilon}_{\varrho}(x)+a^{\varrho\lambda}\partial_{\alpha\lambda}\theta(x/\varepsilon)\partial_{\beta}u^{\varepsilon}_{\varrho}(x),\quad \alpha,\beta =1,2.
\end{equation*}
Obviously, $(w^{\varepsilon}_{\alpha})$ belongs to $H^1_0(\Omega)^2$. From (\ref{boundnorm}) we obtain
\begin{equation}\label{eq:87095}
|w^{\varepsilon}_{\alpha}|_{H^1(\Omega)}\leq C,\quad \varepsilon>0,\quad \alpha=1,2,
\end{equation}
for some constant $C>0$. Therefore, passing to a convergent subsequence denoted by the same index $\varepsilon$, we conclude that there exists functions
\begin{equation*}
w^0_{\alpha}\in H^1_0(\Omega)\quad\text{and}\quad w^1_{\alpha}\in L^2(\Omega; H^1(\dot{\mathcal{Y}})),\quad \alpha=1,2,
\end{equation*}
such that for $\alpha=1,2$,
\begin{equation}\label{eq:87096}
\begin{cases}
w_{\alpha}^{\varepsilon}\rightarrow w_{\alpha}^0\quad\text{weakly in}\quad H^1(\Omega)\quad\text{as}\quad \varepsilon\to 0, &\\
\nabla w_{\alpha}^{\varepsilon}\rightarrow \nabla w_{\alpha}^0+\nabla_y  w_{\alpha}^1\quad\text{two-scale as}\quad  \varepsilon\to 0. &\\
\end{cases}
\end{equation}
Now (\ref{eq:87096}), definition (\ref{eq:87094}) and the standard Lemma on periodic functions [, Theorem 1.8] imply that
\begin{equation*}
w_{\alpha}^{\varepsilon}\rightarrow \partial_{\alpha}u^0_3\quad\text{weakly in}\quad L^2(\Omega)\quad\text{as}\quad \varepsilon\to 0,\quad \alpha=1,2,
\end{equation*}
and the regularity of $u^0_3$ is proved.
\vspace{0.5 cm}
Using (\ref{eq:87094}) again we conclude that for $\alpha=1,2$,
\begin{equation*}
w_{\alpha}^{\varepsilon}=\partial_{\alpha}u^{\varepsilon}_3+a^{\varrho\lambda}\partial_{\alpha\lambda}\theta(x/\varepsilon)u^{\varepsilon}_{\varrho}\rightarrow \partial_{\alpha}u^0_3+\partial^y_{\alpha}u^1_3+a^{\varrho\lambda}\partial_{\alpha\lambda}\theta(y)u^0_{\varrho}\quad\text{two-scale},
\end{equation*}
so that $\partial^y_{\alpha}u^1_3=-a^{\varrho\lambda}\partial_{\alpha\lambda}\theta(y)u^0_{\varrho}$.
\end{proof}

\vspace{0.5 cm}

\begin{lemma}\label{lemmaconv3}
Let $\{\vect{u}^{\varepsilon}\}_{\varepsilon>0}$ be a convergent subsequence as in Lemma \ref{lemmafundtwoscaleconv1}. Then there exists a function $W$ in $L^2(\Omega; H^2(\mathcal{Y}))$ such that for $\alpha=1,2$ and as $\varepsilon$ tends to zero,
\begin{align*}
\partial_{\alpha\beta}u^{\varepsilon}_{3}+\varepsilon^{-1}a^{\varrho\lambda}\partial_{\alpha\beta\lambda}\theta(x/\varepsilon)u^{\varepsilon}_{\varrho}&\rightarrow\partial_{\alpha\beta}u^0_3(x)+\partial^y_{\alpha\beta}W(x,y)+a^{\varrho\lambda}\partial_{\alpha\beta\lambda}\theta(y)u^1_{\varrho}(x,y)\\
&-a^{\varrho\lambda}\partial_{\beta\lambda}\theta(y)\partial_{\alpha}u^0_{\varrho}(x)-a^{\varrho\lambda}\partial_{\alpha\lambda}\theta(y)\partial_{\beta}u^0_{\varrho}(x)-\partial_{\beta}a^{\varrho\lambda}\partial_{\alpha\lambda}\theta(y)u^0_{\varrho}(x),\\
&\text{two-scale}.
\end{align*}
\end{lemma}

\vspace{0.5 cm}

\begin{proof}
Let the auxiliary function $(w^{\varepsilon}_{\alpha})$ be as in (\ref{eq:87094}). As in the proof of Lemma \ref{lemma33} we conclude that there exist functions
\begin{equation*}
w_{\alpha}^0\in H^1_0(\Omega)\quad\text{and}\quad w^1_{\alpha}\in L^2(\Omega; H^1(\dot{\mathcal{Y}})),\quad \alpha=1,2,
\end{equation*} 
such that, up to a subsequence, the converge stated in (\ref{eq:87096}) holds true for $\alpha=1,2$. Moreover,
\begin{equation}\label{eq:87097}
w^0_{\alpha}=\partial_{\alpha}u^0_{3},\quad \alpha=1,2.
\end{equation}
From (\ref{eq:87094}) we also find
\begin{equation}\label{eq:87098}
\partial_{\beta}w^{\varepsilon}_{\alpha}(x)-\partial_{\alpha}w^{\varepsilon}_{\beta}(x)=a^{\varrho\lambda}\partial_{\alpha\lambda}\theta(x/\varepsilon)\partial_{\beta}u^{\varepsilon}_{\varrho}(x)-a^{\varrho\lambda}\partial_{\beta\lambda}\theta(x/\varepsilon)\partial_{\alpha}u^{\varepsilon}_{\varrho}(x),\quad \alpha,\beta=1,2,
\end{equation}
thus, for $\alpha,\beta=1,2$
\begin{equation*}
\partial_{\beta}w^{\varepsilon}_{\alpha}-\partial_{\alpha}w^{\varepsilon}_{\beta}\rightarrow \partial^y_{\beta}w^1_{\alpha}-\partial^y_{\alpha}w^1_{\beta}\quad\text{two-scale as}\,\varepsilon\to 0.
\end{equation*}
On the other hand, by (\ref{boundnorm}) and [2, Theorem 1.8] it holds for $\alpha,\beta=1,2,$
\begin{align*}
a^{\varrho\lambda}\left[\partial_{\alpha\lambda}\theta(x/\varepsilon)\partial_{\beta}u^{\varepsilon}_{\varrho}-\partial_{\beta\lambda}\theta(x/\varepsilon)\partial_{\alpha}u^{\varepsilon}_{\varrho}\right]&\rightarrow a^{\varrho\lambda}\left[\partial_{\alpha\lambda}\theta(y)\partial_{\beta}u^0_{\varrho}-\partial_{\beta\lambda}\theta(y)\partial_{\alpha}u^0_{\varrho}+\partial_{\alpha\lambda}\theta(y)\partial^{y}_{\beta}u^1_{\varrho}\right.\\
&\quad-\left.\partial_{\beta\lambda}\theta(y)\partial^y_{\alpha}u^1_{\varrho}\right]\quad\text{two-scale as}\, \varepsilon\to 0,
\end{align*}
therefore by (\ref{eq:87098}),
\begin{align*}
\partial^y_{\beta}w^1_{\alpha}-\partial^y_{\alpha}w^1_{\beta}&=a^{\varrho\lambda}\left[\partial_{\alpha\lambda}\theta \partial_{\beta}u^0_{\varrho}-\partial_{\beta\lambda}\theta\partial_{\alpha}u^0_{\varrho}+\partial_{\alpha\lambda}\theta(y)\partial^{y}_{\beta}u^1_{\varrho}-\partial_{\beta\lambda}\theta\partial^y_{\alpha}u^1_{\varrho}\right]\\
&=\partial^y_{\beta}\left\{a^{\varrho\lambda}\left[\partial_{\alpha\lambda}\theta u^1_{\varrho}-\partial_{\lambda}\theta \partial_{\alpha}u^0_{\varrho}\right]\right\}-\partial^y_{\alpha}\left\{a^{\varrho\lambda}\left[\partial_{\beta\lambda}\theta u^1_{\varrho}-\partial_{\lambda}\theta \partial_{\beta}u^0_{\varrho}\right]\right\},\,\alpha,\beta=1,2.
\end{align*}
We define the auxiliary functions
\begin{align*}
z_{\alpha}=w^1_{\alpha}-a^{\varrho\lambda}\partial_{\alpha\lambda}\theta u^1_{\varrho}+a^{\varrho\lambda}\partial_{y}\theta \partial_{\alpha}u^0_{\varrho}+\left<a^{\varrho\lambda}\partial_{\alpha\lambda}\theta u^1_{\varrho}\right>,\quad \alpha=1,2,
\end{align*}
where $\left<\cdot\right>$ denotes the mean value of a function over the unit cell $\mathcal{Y}$, i.e., $\left<\phi\right>=\displaystyle\int_{\mathcal{Y}}\phi$. Obviously, 
\begin{equation*}
\vect{z}=(z_{\alpha})\in L^2(\Omega; H^1(\dot{\mathcal{Y}}))^2\quad\text{and}\quad \text{curl}_{y}\vect{z}=0,
\end{equation*}
thus by [1, Lemma 3.3] there exist a function $W$ such that
\begin{equation*}
W\in L^2(\Omega; H^2(\dot{\mathcal{Y}}))\quad \text{and}\quad \vect{z}=\nabla_y W.
\end{equation*}
Using this, we find
\begin{equation*}
\partial^y_{\beta}w^1_{\alpha}=a^{\varrho\lambda}\partial_{\alpha\beta\lambda}\theta u^1_{\varrho}+a^{\varrho\lambda}\partial_{\alpha\lambda}\theta\partial^y_{\beta}u^1_{\varrho}-a^{\varrho\lambda}\partial_{\beta\lambda}\theta \partial_{\alpha}u^0_{\varrho}+\partial^y_{\alpha\beta}W.
\end{equation*}
Recalling the definition (\ref{eq:87094}) of $w^{\varepsilon}_{\alpha}$ and (\ref{eq:87096}) we obtain for $\alpha,\beta=1,2,$ the following two-scale convergences:
\begin{align*}
&\partial_{\alpha\beta}u^{\varepsilon}_3+\varepsilon^{-1}a^{\varrho\lambda}\partial_{\alpha\beta\lambda}\theta u^{\varepsilon}_{\varrho}=\partial_{\beta}w^{\varepsilon}_{\alpha}-a^{\varrho\lambda}\partial_{\alpha\lambda}\theta \partial_{\beta}u^{\varepsilon}_{\varrho}-\partial_{\beta}a^{\varrho\lambda}\partial_{\alpha\lambda}\theta u^{\varepsilon}_{\varrho}\\
&\rightarrow \partial_{\alpha\beta}u^0_3+\partial^y_{\alpha\beta}W+a^{\varrho\lambda}\partial_{\alpha\beta\lambda}\theta u^1_{\varrho}-\left(a^{\varrho\lambda}\partial_{\beta\lambda}\theta\partial_{\alpha}u^0_{\varrho}+a^{\varrho\lambda}\partial_{\alpha\lambda}\theta\partial_{\beta}u^0_{\varrho}+\partial_{\beta}a^{\varrho\lambda}\partial_{\alpha\lambda}\theta u^0_{\varrho}\right)\quad\text{as}\, \varepsilon\to 0.
\end{align*}

\end{proof}

\vspace{0.5 cm}

This is our first convergence theorem.
\vspace{0.5 cm}

\begin{teo}\label{twoscaleourtheorem}
Assume that (\ref{boundforce}) holds true. Then there exists a subsequence of $\{\vect{u}^{\varepsilon}\}_{\varepsilon>0}$ denoted again by $\{\vect{u}^{\varepsilon}\}_{\varepsilon>0}$, and functions
\begin{equation*}
\begin{cases}
u^0_{\alpha}\in H^1_0(\Omega),\quad \alpha=1,2,\quad u^0_3\in H^2_0(\Omega), &\\
u^1_{\alpha}\in L^2(\Omega; H^1(\dot{\mathcal{Y}})), \quad \alpha=1,2,\quad\text{and}\quad W\in L^2(\Omega; H^2(\dot{\mathcal{Y}})),
\end{cases}
\end{equation*}
such that as $\varepsilon$ tends to zero, for $i=1,2,3$ and $\alpha,\beta=1,2,$
\begin{align*}
u^{\varepsilon}_{i}\rightarrow u^0_i\quad\text{weakly in}\, H^1(\Omega),\quad& \nabla u^{\varepsilon}_{\alpha}\rightarrow \nabla u^0_{\alpha}+\nabla_y u^1_{\alpha}\quad\text{two-scale,}\\
\partial_{\alpha}u^{\varepsilon}_3\rightarrow \partial_{\alpha}u^0_3-a^{\varrho\lambda}\partial_{\alpha\lambda}\theta(y) u^0_{\varrho}\quad&\text{two-scale,}\\
\partial_{\alpha\beta}u^{\varepsilon}_{3}+\varepsilon^{-1}a^{\varrho\lambda}\partial_{\alpha\beta\lambda}\theta(x/\varepsilon)u^{\varepsilon}_{\varrho}&\rightarrow\partial_{\alpha\beta}u^0_3+\partial^y_{\alpha\beta}W+a^{\varrho\lambda}\partial_{\alpha\beta\lambda}\theta u^1_{\varrho}\\
&\quad-a^{\varrho\lambda}\partial_{\beta\lambda}\theta \partial_{\alpha}u^0_{\varrho}-a^{\varrho\lambda}\partial_{\alpha\lambda}\theta \partial_{\beta}u^0_{\varrho}-\partial_{\beta}a^{\varrho\lambda}\partial_{\alpha\lambda}\theta u^0_{\varrho}\quad\text{two-scale}.
\end{align*}
\end{teo}

\begin{proof}
The proof is contained in Lemmas \ref{lemmafundtwoscaleconv1}-\ref{lemmaconv3}.
\end{proof}

\vspace{1.0 cm}

\subsection{Two-scale problem}
\vspace{0.5 cm}

In the sequel we assume that the right hand sides of $\varepsilon-$problems (\ref{cartesian}) satisfy
\begin{equation}\label{weakconvergenceforce}
\vect{f}_{\varepsilon}\rightarrow \vect{f}_0\quad \text{weakly in}\, L^2(\Omega)^3\,\,\text{as}\,\,\varepsilon\to 0.
\end{equation}

This assmption eneables us to prove convergence in certain sense of $\{\vect{u}_{\varepsilon}\}_{\varepsilon>0}$ to a solution of the so called two-scale problem posed over the function space
\begin{equation}\label{spacefunctiontwsclimit}
\vect{X}=H^1_0(\Omega)^2\times H^2_0(\Omega)\times L^2(\Omega;\dot{H}^1(\mathcal{Y}))^2\times L^2(\Omega;\dot{H}^2(\mathcal{Y}))
\end{equation}

\vspace{0.5 cm}

We also  introduce some new notations; the indices $\alpha,\beta=1,2$, while $(\vect{v}, \vect{v}^1, V)$ belongs to $\vect{X}$:

\begin{align}
\gamma_{\alpha\beta}^{0}(\vect{v}, \vect{v}^1)&\label{gammaexpansion2021}=\gamma_{\alpha\beta}(\vect{v})+e^{y}_{\alpha\beta}(\vect{v}^1),\quad \text{with}\quad e^{y}_{\alpha\beta}(\vect{v}^1)=\displaystyle\frac{1}{2}\left(\partial^{y}_{\alpha}v^1_{\beta}+\partial^{y}_{\beta}v^1_{\alpha}\right),\\
M^{y}_{\alpha\beta}(\vect{v})&\nonumber=-a^{\varrho\lambda}\partial_{\lambda\alpha}\theta(y)\Gamma^{\sigma}_{\varrho\beta}v_{\sigma}+\partial_{\alpha}a^{\varrho\lambda}\partial_{\lambda\beta}\theta(y)v_{\varrho}-\partial_{\beta}a^{\varrho\lambda}\partial_{\alpha\lambda}\theta(y)v_{\varrho}\\
&\label{nexpansion1}\quad+\displaystyle\frac{\left(\partial_{\alpha2}\theta(y) \vect{k}_{13}+\partial_{\alpha1}\theta(y) \vect{k}_{32}\right)}{\sqrt{a}}a^{\varrho\lambda}\partial_{\lambda}a_{\beta}v_{\varrho},\\
N^{y}_{\alpha\beta}(\vect{v}^1)&\label{nexpansiony2}=a^{\varrho\lambda}\partial_{\alpha\lambda\beta}\theta(y)v^1_{\varrho}+a^{\varrho\lambda}[\partial_{\lambda\beta}\theta(y)\partial^{y}_{\alpha}v^1_{\varrho}+\partial_{\lambda\alpha}\theta(y)\partial_{\beta}^{y}v^1_{\varrho}]+b^{\varrho}_{\beta}\partial_{\alpha}^yv^1_{\varrho}+b^{\varrho}_{\alpha}\partial_{\beta}^yv^1_{\varrho},\\
\Gamma_{\alpha\beta}^0(\vect{v}, \vect{v}^1, V)&\label{Gammaexpansion20212}=\Gamma_{\alpha\beta}(\vect{v})+\partial^{y}_{\alpha\beta}V+N^y_{\alpha\beta}(\vect{v}^1)+M^{y}_{\alpha\beta}(\vect{v}).
\end{align}

\vspace{0.5 cm}

\begin{lemma}\label{lemma:convtensor2021}
There exists a subsequence of the sequence $\{\vect{u}^{\varepsilon}\}_{\varepsilon>0}$, denoted by the same index $\varepsilon$, and a function $[\vect{u}^0=(u_i^0), \vect{u}^1=(u^1_{\alpha}), W]$ in $\vect{X}$ such that for $\alpha, \beta=1,2,$
\begin{align}
\gamma^{\varepsilon}_{\alpha\beta}(\vect{u}^{\varepsilon})&\label{twoscalegammalimit}\overset{2}{\rightharpoonup}\gamma_{\alpha\beta}^0(\vect{u}^0,\vect{u}^1)-\partial_{\alpha\beta}\theta(y)u^0_3\quad\text{two-scale as}\,\,\varepsilon\to 0,\\
\Gamma_{\alpha\beta}^{\varepsilon}(\vect{u}^{\varepsilon})&\nonumber\overset{2}{\rightharpoonup}\Gamma_{\alpha\beta}^0(\vect{u}^0,\vect{u}^1,W)-a^{\lambda\varrho}\partial_{\varrho\alpha}\theta(y)\partial_{\lambda\beta}\theta(y)u^0_3\\
&\label{twoscalecapGammalimit}\quad-\left[b^{\lambda}_{\alpha}\partial_{\lambda\beta}\theta(y)+b^{\varrho}_{\beta}\partial_{\varrho\alpha}\theta(y)\right]u^0_3\quad\text{two-scale as}\,\,\varepsilon\to 0.
\end{align}

\end{lemma}

\vspace{0.5 cm}

\begin{proof}
The proof follows directly of (\ref{weakconvergenceforce}), Theorem \ref{teo:3.4}, asymptotic expansions in Theorem \ref{teo:332} for $\gamma^{\varepsilon}_{\alpha\beta}(\vect{u})$ and $\Gamma^{\varepsilon}_{\alpha\beta}(\vect{u})$ and definitions given in (\ref{gammaexpansion2021}) and (\ref{Gammaexpansion20212}).
\end{proof}

\vspace{0.5 cm}

For such tripples and for functions $(\vect{z}, \vect{z}^1, Z)$ belonging to $\vect{X}$ we define the form $B^0(\cdot, \cdot)$ by

\begin{align}
B^0&\nonumber\left[(\vect{z}, \vect{z}^1, Z), (\vect{\psi}, \vect{\varphi}, \varphi_3)\right]\\
&\label{ourlimitproblem}=\displaystyle\int_{\Omega}\int_{\mathcal{Y}}a^{\alpha\beta\varrho\sigma}\left[d\gamma^0_{\varrho\sigma}(\vect{z}, \vect{z}^1)\gamma^0_{\alpha\beta}(\vect{\psi}, \vect{\varphi})+\displaystyle\frac{d^3}{3}\Gamma^0_{\varrho\sigma}(\vect{z}, \vect{z}^1, Z)\Gamma^0_{\alpha\beta}(\vect{\psi}, \vect{\varphi}, \varphi_3)\right]\sqrt{a}\, dy\,dx,
\end{align}
where
\begin{equation}\label{elasticitytensor}
a^{\alpha\beta\varrho\sigma}=\displaystyle\frac{4\lambda\mu}{\lambda+2\mu}a^{\alpha\beta}a^{\varrho\sigma}+2\mu\left[a^{\alpha\varrho}a^{\beta\sigma}+a^{\alpha\sigma}a^{\beta\varrho}\right].
\end{equation}

\vspace{0.5 cm}

Now we define the two-scale problem as the following boundary-value problem written in a week form:\\
for given $\vect{f}^0$ in $L^2(\Omega)^3$ find $(\vect{u}^0,\widetilde{\vect{u}}^1, U)\in \vect{X}$ such that
\begin{equation}\label{twoscaleproblemnew}
B^0[(\vect{u}^0,\widetilde{\vect{u}}^1, U),(\vect{z}, \vect{z}^1, Z)]=\displaystyle\int_{\Omega}\vect{f}^0\cdot\vect{z}\sqrt{a}\,dx,\quad\forall (\vect{z}, \vect{z}^1, Z)\in\vect{X}.
\end{equation}

\vspace{1.0 cm}

\section{Decoupling of the two-scale problem}

\vspace{0.5 cm}

\begin{lemma}\label{ourresult1}
Function $(u^0_{i})$ is unique solution in $H^1_0(\Omega)^2\times H^2_0(\Omega)$ of the following variational equation:
\begin{align}
&\nonumber\displaystyle\int_{\Omega}a^{\alpha\beta\varrho\sigma}\left\{d\gamma_{\varrho\sigma}(\vect{u}^0)\gamma_{\alpha\beta}(\vect{z})+\displaystyle\frac{d^3}{3}\Gamma_{\varrho\sigma}(\vect{u}^0)\Gamma_{\alpha\beta}(\vect{z})+\displaystyle\frac{d^3}{3}\int_{\mathcal{Y}}M^y_{\varrho\sigma}(\vect{u}^0)M^y_{\alpha\beta}(\vect{z})\,dy\right\}\sqrt{a}\,\,dx\\
\quad &\label{ourtwoscaleproblem1}=\displaystyle\int_{\Omega}\vect{f}^0\cdot\vect{z}\sqrt{a}\,dx,\,\forall \vect{z}\in H^1_0(\Omega)^2\times H^2_0(\Omega).
\end{align}
\end{lemma}

\vspace{0.5 cm}

The proof follows directly from (\ref{twoscaleproblemnew}) by the following choice of test-functions:
\begin{equation*}
\vect{z}^1=Z=0.
\end{equation*}

\vspace{0.5 cm}
The rest of the variational equation (\ref{twoscaleproblemnew}) for unknowns $(\widetilde{\vect{u}}^1, U)$ becomes
\begin{align}
&\nonumber\displaystyle\int_{\Omega}\int_{\mathcal{Y}}a^{\alpha\beta\varrho\sigma}\left\{d\, e^y_{\varrho\sigma}(\widetilde{\vect{u}}^1)e^y_{\alpha\beta}(\vect{z}^1)+\displaystyle\frac{d^3}{3}\left[\partial^y_{\varrho\sigma}U\,+N^y_{\varrho\sigma}(\widetilde{\vect{u}}^1)\right]\left[\partial^y_{\alpha\beta}Z\,+N^y_{\alpha\beta}(\vect{z}^1)\right]\right\}\sqrt{a}=0,\\
&\label{freeterms}\,\forall (\vect{z}^1,Z)\in L^2(\Omega;\dot{H}^1(\mathcal{Y}))^2\times L^2(\Omega;\dot{H}^2(\mathcal{Y})).
\end{align}

\vspace{1.0 cm}

Let us introduce the function space
\begin{equation}\label{spacefordecouplesys}
\mathcal{H}=(\dot{H}^1(\mathcal{Y}))^2\times \dot{H}^2(\mathcal{Y}).
\end{equation}
$\mathcal{H}$ is a Hilbert space endowed with the norm
\begin{equation*}
\left|\left|\left(\vect{v}, V\right)\right|\right|_{\mathcal{H}}=\left(|\vect{v}|^2_{H^1(\mathcal{Y})^2}+\displaystyle\sum_{\alpha, \beta=1}^{2}|\partial_{\alpha\beta}V|^2_{L^2(\mathcal{Y})}\right)^{1/2}.
\end{equation*}

We define  the bilinear form $\mathcal{B}:\mathcal{H}\times\mathcal{H}\to\RR $

\begin{align}
\mathcal{B}\left[\left(\vect{v}, V\right),\left(\vect{z}, Z\right)\right]&\nonumber=\displaystyle\int_{\mathcal{Y}}a^{\alpha\beta\varrho\sigma}\left\{\right.d\, e^y_{\varrho\sigma}(\vect{v})e^y_{\alpha\beta}(\vect{z})+\displaystyle\frac{d^3}{3}\left[\partial^y_{\varrho\sigma}V\,+N^y_{\varrho\sigma}(\vect{v})\right]\\
&\label{equationdecouplesys1}\displaystyle\left.\times\left[\partial^y_{\alpha\beta}Z\,+N^y_{\alpha\beta}(\vect{z})\right]\right\}\sqrt{a}.
\end{align}

For $\zeta, \eta\in\{1,2\}$ and $\left(\vect{z}, Z\right)\in\mathcal{H}$ we set

\begin{equation}\label{eqcdecuple2}
F_{\xi\eta}\left(\vect{z}, Z\right)=\displaystyle\int_{\mathcal{Y}}\frac{d^3}{3}a^{\alpha\beta\xi\eta}\left[a^{\varrho\lambda}\partial_{\xi\lambda\eta}\theta+\partial_{\xi}b^{\varrho}_{\eta}+\partial_{\eta}b^{\varrho}_{\xi}\right]z_{\varrho}
\end{equation}
$F_{\xi\eta}$ is linear and bounded functional on $\mathcal{H}$. It follows from periodicity of the test functions that $F_{\xi\eta}$ can be written in the form
\begin{equation*}
F_{\xi\eta}\left(\vect{z}, Z\right)=-\displaystyle\int_{\mathcal{Y}}\frac{d^3}{3}a^{\alpha\beta\xi\eta}\left[N^y_{\xi\eta}(\vect{z})\right].
\end{equation*}
Now for $\xi, \eta\in\{1,2\}$ we define the following local problems:

\begin{equation}\label{localproblems}
\begin{cases}
\left(\widetilde{\vect{\Phi}}^{\xi\eta}, \Phi^{\xi\eta}\right)\in \mathcal{H}, &\\
\mathcal{B}\left[\left(\widetilde{\vect{\Phi}}^{\xi\eta}, \Phi^{\xi\eta}\right), (\vect{z}, Z)\right]=F_{\xi \eta}(\vect{z}, Z),\quad (\vect{z}, Z)\in \mathcal{H}. & 
\end{cases}
\end{equation}

\vspace{1.0 cm}

\begin{teo}\label{teobilinearform1}
For any $\xi, \eta\in\{1,2\}$ problem (\ref{localproblems}) has a unique solution
\begin{equation*}
\left(\widetilde{\vect{\Phi}}^{\xi\eta}, \Phi^{\xi\eta}\right)\in \mathcal{H}\cap \left[\dot{H}^2(\mathcal{Y})^2\times \dot{H}^4(\mathcal{Y})\right].
\end{equation*}
\end{teo}

\begin{proof}
The proof follows from an application of the Lax-Milgram theorem which guarantee the existence and the uniqueness of solution. To prove $\mathcal{H}-$ellipticity of $\mathcal{B}$ we can use the Korn inequality. 
\end{proof}

\vspace{1.0 cm}

\section*{Acknowledgments}

The first author was supported partially by Agencia Nacional de Investigaci\'on y Desarrollo de Chile (ANID) through Fondecyt Postdoctorado 2023 Grant No. 3230202 and also thanks the generous hospitality provided by the Department of Mathematics and Mathematical Statistics at Ume\r{a} University of Sweden, particularly by  Sebastian Throm,  \r{A}ke Br\"annstr\"om and Konrad Abramowicz.

\vspace{1.5cm}

\end{document}